\documentclass[12pt,a4paper,reqno]{amsart}
\usepackage{amssymb}
\usepackage{amscd}
\numberwithin{equation}{section}

\theoremstyle{plain}

\newtheorem{theorem}[subsection]{Theorem}
\newtheorem{proposition}[subsection]{Proposition}
\newtheorem{lemma}[subsection]{Lemma}
\newtheorem{corollary}[subsection]{Corollary}
\newtheorem{conjecture}[subsection]{Conjecture}

\theoremstyle{definition}

\newtheorem{definition}[subsection]{Definition}

\newtheorem{remark}[subsection]{Remark}
\newtheorem{example}[subsection]{Example}

\renewcommand{\leq}{\leqslant}
\renewcommand{\geq}{\geqslant}

\newsavebox{\proofbox}
\savebox{\proofbox}{\begin{picture}(7,7)%
  \put(0,0){\framebox(7,7){}}\end{picture}}

\newcommand\E{\mathbb{E}}
\newcommand\Z{\mathbb{Z}}
\newcommand\R{\mathbb{R}}

\newcommand\C{\mathbb{C}}
\newcommand\N{\mathbb{N}}
\newcommand\X{\mathbf{X}}

\newcommand\D{\mathcal{D}}
\newcommand\B{\mathcal{B}}

\newcommand\F{\mathbb{F}}

\newcommand\eps{\varepsilon}

\newcommand\Lip{\operatorname{Lip}}

\newcommand\Poly{\operatorname{Poly}}

\newcommand\Phase{{\mathcal P}} 
\newcommand\charac{\operatorname{char}}
\renewcommand\th{{\operatorname{th}}}
\newcommand\ader{{\Delta}}
\newcommand\mder{{\Delta\!\!\!\!\!\hbox{\raisebox{0.2ex}{\tiny\ \textbullet}}\ \!}}

\begin{document}

\title[Finite fields inverse Gowers norm conjecture]{The inverse conjecture for the Gowers norm over finite fields via the correspondence principle}

\author{Terence Tao}
\address{UCLA Department of Mathematics, Los Angeles, CA 90095-1596.
}
\email{tao@math.ucla.edu}

\author{Tamar Ziegler}
\address{Department of Mathematics, Technion, Haifa, Israel 32000}
\email{tamarzr@tx.technion.ac.il}

\thanks{The first author is supported by a grant from the MacArthur Foundation, and by NSF grant CCF-0649473. 
The second author is supported by ISF grant 557/08,  by a Landau fellowship - supported by the Taub foundations, and by an Alon fellowship.}


\begin{abstract}  The inverse conjecture for the Gowers norms $U^d(V)$ for finite-dimensional vector spaces $V$ over a finite field $\F$ asserts, roughly speaking, that a bounded function $f$ has large Gowers norm $\|f\|_{U^d(V)}$ if and only if it correlates with a phase polynomial $\phi = e_\F(P)$ of degree at most $d-1$, thus $P: V \to \F$ is a polynomial of degree at most $d-1$.  In this paper, we develop a variant of the Furstenberg correspondence principle which allows us to establish this conjecture in the large characteristic case $\charac(F) \geq d$ from an ergodic theory counterpart, which was recently established by Bergelson and the authors in \cite{berg}. In low characteristic we obtain a partial result, in which the phase polynomial $\phi$ is allowed to be of some larger degree $C(d)$.  The full inverse conjecture remains open in low characteristic; the counterexamples in \cite{finrat}, \cite{lms} in this setting can be avoided by a slight reformulation of the conjecture. 
\end{abstract}

\maketitle

\section{Introduction}

\subsection{The combinatorial inverse conjecture in finite characteristic}

Let $\F$ be a finite field of prime order.  Throughout this paper, $\F$ will be considered fixed (e.g. $\F = \F_2$ or $\F = \F_3$), and the term ``vector space'' will be shorthand for ``vector space over $\F$'', and more generally any linear algebra term (e.g. span, independence, basis, subspace, linear transformation, etc.) will be understood to be over the field $\F$.  

If $V$ is a vector space, $f: V \to \C$ is a function, and $h \in V$ is a shift, we define the (multiplicative) derivative $\mder_h f: V \to \C$ of $f$ by the formula
$$ \mder_h f := (T_h f) \overline{f}$$
where the shift operator $T_h$ with shift $h$ is defined by $T_h f(x) := f(x+h)$.
An important special case arises when $f$ takes the form $f = e_\F(P)$, where $P: V \to \F$ is a function, and $e_\F: \F \to \C$ is the standard character $e_\F(j) := e^{2\pi i j/|\F|}$ for $j=0,\ldots,|\F|-1$.  In that case we see that $\mder_h f = e_\F( \ader_h P )$, where $\ader_h P: V \to \F$ is the (additive) derivative of $P$, defined as
$$ \ader_h P = T_h P - P.$$

Given an integer $d \geq 0$, we say that a function $P: V \to \F$ is a \emph{polynomial of degree at most $d$} if we have $\ader_{h_1} \ldots \ader_{h_{d+1}} P = 0$ for all $h_1,\ldots,h_{d+1} \in V$, and write $\Poly_d(V)$ for the set of all polynomials on $V$ of degree at most $d$; thus for instance $\Poly_0(V)$ is the set of constants, $\Poly_1(V)$ is the set of linear polynomials on $V$, $\Poly_2(V)$ is the set of quadratic polynomials, and so forth.  It is easy to see that $\Poly_d(V)$ is a vector space, and if $V = \F^n = \{ (x_1,\ldots,x_n): x_1,\ldots,x_n \in \F\}$ is the standard $n$-dimensional vector space, then $\Poly_d(V)$ has the monomials $x_1^{i_1} \ldots x_n^{i_n}$ for $0 \leq i_1,\ldots, i_n < |\F|$ and $i_1+\ldots+i_n \leq d$ as a basis\footnote{The restriction $i_1,\ldots,i_n <|\F|$ arises of course from the identity $x^{|\F|} = x$ for all $x \in \F$.}.  

We shall say that a function $f: V \to \C$ is a \emph{phase polynomial of degree at most $d$} if all $(d+1)^{\th}$ multiplicative derivatives $\mder_{h_1} \ldots \mder_{h_{d+1}} f$ are identically $1$, and write $\Phase_d(V)$ for the space of all phase polynomials of degree at most $d$.  We have the following equivalence between polynomials and phase polynomials in the high characteristic case:

\begin{lemma}[Phase polynomials are exponentials of polynomials]\label{equiv-lem} Suppose that $0 \leq d < \charac(\F)$, and $f: V \to \C$.  Then the following are equivalent:
\begin{itemize}
\item[(i)] $f \in \Phase_d(V)$.
\item[(ii)] $f = e^{2\pi i\theta} e_\F(P)$ for some $\theta \in \R/\Z$ and $P \in \Poly_d(V)$.
\end{itemize}
\end{lemma}

\begin{proof} See \cite[Lemma D.5]{berg}.  
\end{proof}

\begin{remark}\label{fail} The lemma fails in the low characteristic case $d \geq \charac(\F)$; consider for instance the function $f: \F_2 \to \C$ defined by $f(1) := i$ and $f(0) := 1$.  This function lies in $\Phase_2(\F_2)$ but does not arise from a polynomial in $\Poly_2(\F_2)$.
\end{remark}

\begin{definition}[Expectation notation]  If $A$ is a finite non-empty set and $f: A \to \C$ is a function, we write $|A|$ for the cardinality of $A$, and $\E_A f$, $\int_A f$, or $\E_{x \in A} f(x)$ for the average $\frac{1}{|A|} \sum_{x \in A} f(x)$.  
\end{definition}

\begin{definition}[Gowers uniformity norm]\cite{gowers-4}, \cite{gowers}  Let $V$ be a finite vector space, let $f: V \to \C$ be a function, and let $d \geq 1$ be an integer.  We then define the \emph{Gowers norm} $\|f\|_{U^{d}(V)}$ of $f$ to be the quantity
$$ \|f\|_{U^{d}(V)} := |\E_{h_1,\ldots,h_d} \int_V \mder_{h_1} \ldots \mder_{h_{d}} f|^{1/2^{d}},$$
thus $\|f\|_{U^{d+1}(V)}$ measures the average bias in $d^\th$ multiplicative derivatives of $f$.
We also define the \emph{weak Gowers norm} $\|f\|_{u^{d}(V)}$ of $f$ to be the quantity
\begin{equation}\label{udeq-eq} \|f\|_{u^{d}(V)} := \sup_{\phi \in \Phase_{d-1}(V)} |\int_V f \overline{\phi}|,
\end{equation}
thus $\|f\|_{u^{d}(V)}$ measures the extent to which $f$ can correlate with a phase polynomial of degree at most $d-1$.
\end{definition}

\begin{remark}\label{efc}  
It can in fact be shown that the Gowers and weak Gowers norm are in fact norms for $d \geq 2$ (and seminorms for $d=1$), see e.g. \cite{gowers} or \cite{tao-vu}.  Further discussion of these two norms can be found in \cite{gt:inverse-u3}.  In view of Lemma \ref{equiv-lem}, in the high characteristic case $\charac(F) \geq d$ one can replace the phase polynomial $\phi \in \Phase_{d-1}(V)$ in \eqref{udeq-eq} by the exponential $e_\F(P)$ of a polynomial $P \in \Poly_{d-1}(V)$.  However, this is not the case in low characteristic.  For instance, let $\F = \F_2$, $V = \F_2^n$, and consider the symmetric function $S_4: V \to \F_2$ defined by
$$ S_4(x_1,\ldots,x_n) := \sum_{1 \leq i < j < k < l \leq n} x_i x_j x_k x_l.$$
Then the function $f := (-1)^{S_4}$ has low correlation with any exponential $e_\F(P) = (-1)^P$ of a cubic polynomial $P \in \Poly_3(V)$ in the sense that $\E_{x \in V} f e_\F(-P) = o_{n \to \infty}(1)$ (see \cite{lms}, \cite{finrat}); on the other hand, it is not hard to verify that the function 
$$g(x_1,\ldots,x_n) := e^{2\pi i |x|/8 },$$
where $|x|$ denotes the number of indices $1 \leq j \leq n$ for which $x_j=1$, lies in $\Phase_3(V)$ and has a large inner product with $f$; indeed, since $f(x) = +1$ when $|x| = 0,1,2,3 \mod 8$ and $-1$ otherwise, we easily check that
\begin{align*}
\E_{x \in V} f \overline{g} &= \frac{1}{8} (1 + e^{\frac{-2\pi i} {8}} + e^{\frac{-4\pi i}{8}} + e^{\frac{-6\pi i}{ 8}} - e^{\frac{-8\pi i}{8}} - e^{\frac{-10\pi i}{ 8}} - e^{\frac{-12 \pi i}{8}} - e^{\frac{-14 \pi i}{8}}) + o_{n \to \infty}(1) \\
&= \frac{1-i-\sqrt{2} i}{4} + o_{n \to \infty}(1).
\end{align*}
In particular, we see that $\|(-1)^{S_4} \|_{u^4(V)}$ is bounded from below by a positive absolute constant for large $n$.
\end{remark}

Let $\D := \{ z \in \C: |z| \leq 1 \}$ be the compact unit disk.  This paper is concerned with the following conjecture:

\begin{conjecture}[Inverse Conjecture for the Gowers norm]\label{mainconj}  Let $\F$ be a finite field and let $d \geq 1$ be an integer.  Then for every $\delta > 0$ there exists $\eps > 0$ such that $\|f\|_{u^{d}(V)} \geq \eps$ for every finite vector space $V$ and every function $f: V \to \D$ such that $\|f\|_{U^{d}(V)} \geq \delta$.
\end{conjecture}

\begin{remark} This result is trivial for $d=1$, and follows easily from Plancherel's theorem for $d=2$.  The result was established for $d=3$ in \cite{gt:inverse-u3} (for odd characteristic) and \cite{sam} (for even characteristic), and a formulation of Theorem \ref{main} was then conjectured in both papers, in which the phase polynomials were constrained to be $\charac(\F)^\th$ roots of unity.  This formulation of the conjecture turned out to fail in the low characteristic regime $ \charac(\F)+1<d$ (see \cite{finrat}, \cite{lms}); however, the counterexamples given there do not rule out the conjecture as formulated above in this case, basically because of the discussion in Remark \ref{efc}. 

The case when $\delta$ was sufficiently close to $1$ (depending on $d$) was treated in \cite{akklr}, while the case when $\charac(\F)$ is large compared to $d$ and $\delta$ was established in \cite{stv}.  In \cite{finrat}, Theorem \ref{main} was also established in the case when $f$ was a phase polynomial of degree less than $\charac(\F)$.  These results have applications to solving linear systems of equations (and in particular, in finding arithmetic progressions) in subsets of vector spaces \cite{green-tao-finfieldAP4s}, \cite{wolf} and also to polynomiality testing \cite{sam}, \cite{bv}.  Conjecture \ref{mainconj} is also the finite field analogue of a corresponding inverse conjecture for the Gowers norm in cyclic groups $\Z/N\Z$, which is of importance in solving linear systems of equations in sets of integers such as the primes; see \cite{green-tao-linearprimes}, \cite{fhk} for further discussion.  
\end{remark}

The main result of this paper is to establish this conjecture in the high characteristic case.

\begin{theorem}[Inverse Conjecture for the Gowers norm in high characteristic]\label{main}  Conjecture \ref{mainconj} holds whenever $\charac(\F) \ge d$.
\end{theorem}

In the low characteristic case we have a partial result:

\begin{theorem}[Partial inverse Conjecture for the Gowers norm]\label{main2}  Let $\F$ be a finite field and let $d \geq 1$ be an integer.  Then for every $\delta > 0$ there exists $\eps > 0$ such that $\|f\|_{u^{k}(V)} \geq \eps$ for every finite vector space $V$ and every function $f: V \to \D$ such that $\|f\|_{U^{d}(V)} \geq \delta$, where $k=C(d)$ depends only on $d$.
\end{theorem}

\begin{remark}  One could in principle make the quantity $k=C(d)$ in Theorem \ref{main2} explicit, but this would require analyzing the arguments in \cite{berg} in careful detail.  One should however be able to obtain reasonable values of $k$ for small $d$ (e.g. $d=4$).
\end{remark}

The proofs of Theorems \ref{main}, \ref{main2} rely on four additional ingredients:
\begin{itemize}
\item An ergodic inverse theorem for the Gowers norm for $\F^\omega$-systems (Theorems \ref{ergmain-thm}, \ref{ergmain-thm2}), established in \cite{berg};
\item The Furstenberg correspondence principle\cite{furst}, combined with the random averaging trick of Varnavides\cite{varnavides};
\item A statistical sampling lemma (Proposition \ref{hhh}); and
\item Local testability of phase polynomials (Lemma \ref{polytest-lem}), essentially established in \cite{akklr}.
\end{itemize}

Of these ingredients, the ergodic inverse theorem is the most crucial, and we now pause to describe it in detail.

\subsection{The ergodic inverse conjecture in finite characteristic}

Let $\F^\omega := \bigcup_{n=0}^\infty \F^n$ be the inverse limit of the finite-dimensional vector spaces $\F^n$, where each $\F^n$ is included in the next space $\F^{n+1}$ in the obvious manner; equivalently, $\F^\omega$ is the space of sequences $(x_i)_{i=1}^\infty$ with $x_i \in \F$, and all but finitely many of the $x_i$ non-zero.  This is a countably infinite vector space over $\F$.  

\begin{definition}[$\F^\omega$-system]  A \emph{$\F^\omega$-system} is a quadruplet $\X = (X,\B,\mu,(T_g)_{g \in \F^\omega})$, where $(X,\B,\mu)$ is a probability space, and $T: h \mapsto T_h$ is a measure-preserving action of $\F^\omega$ on $X$, thus for each $h \in \F^\omega$, $T_h: X \to X$ is a measure-preserving bijection such that $T_h \circ T_k = T_{h+k}$ for all $h,k \in \F^\omega$.  Given any measurable $\phi: X \to \C$ and $h \in \F^\omega$, we define $T_h \phi: X \to \C$ to be the function $T_h \phi := \phi \circ T_{h}$, and $\mder_h \phi: X \to \C$ to be the function $\mder_h \phi := T_h \phi \cdot\overline{\phi}$.  We also define the inner product $\langle f, g \rangle := \int_X f \overline{g}\ d\mu$ for all $f, g \in L^2(\X)$, where the Lebesgue spaces $L^p(\X) = L^p(X,\B,\mu)$ are defined in the usual manner.  We say that the system is \emph{ergodic} if the only $\F^\omega$-invariant functions on $L^2(\X)$ are the constants.
\end{definition}

\begin{definition}[Phase polynomial]  Let $\X = (X,\B,\mu,(T_g)_{g \in \F^{\omega}})$ be an $\F^\omega$-system, and let $d \geq 0$.  We say that a function $\phi \in L^\infty(\X)$ is a \emph{phase polynomial of degree at most $d$} if we have $\mder_{h_1} \ldots\mder_{h_{d+1}} \phi = 1$ $\mu$-a.e. for all $h_1,\ldots,h_{d+1} \in \F^\omega$.  We let $\Phase_d(\X)$ denote the space of all phase polynomials.  
\end{definition}

\begin{remark} By setting $h_1=\ldots=h_{d+1}=0$ we see that every phase polynomial $\phi \in \Phase_d(\X)$ has unit magnitude: $|\phi|=1$ $\mu$-a.e..
\end{remark}

\begin{definition}[Gowers-Host-Kra seminorms]\cite{host-kra}  Let $\X = (X,\B,\mu,(T_g)_{g \in \F^{\omega}})$ be a $\F^\omega$-system, and let $\phi \in L^\infty(\X)$.  We define the \emph{Gowers-Host-Kra seminorms} $\|\phi\|_{U^d(\X)}$ for $d \geq 1$ recursively as follows:
\begin{itemize}
\item If $d=1$, then $\| \phi \|_{U^1(\X)} := \limsup_{n \to \infty} \left(\| \E_{h \in \F^n} T_h \phi \|_{L^2(X,\mu)}^2\right)^{1/2}$;
\item If $d>1$, then $\| \phi \|_{U^d(\X)} := \limsup_{n \to \infty} \left(\| \mder_h \phi \|_{U^{d-1}(X,\mu,T)}^{2^{d-1}} \right)^{1/2^d}$.
\end{itemize}
We also define the \emph{weak Gowers-Host-Kra seminorm} $\|\phi\|_{u^d(\X)}$ as
$$ \|\phi\|_{u^d(\X)} := \sup_{\psi \in \Phase_{d-1}(\X)} |\langle \phi, \psi \rangle|.$$
\end{definition}

\begin{example} If $\phi \in \Phase_{d-1}(\X)$ is a phase polynomial of degree at most $d-1$, then $\| \phi \|_{U^d(\X)} = \| \phi \|_{u^d(\X)} = 1$.  
\end{example}

\begin{remark} One can use the ergodic theorem to show that the limits here in fact converge, but we will not need this.  The $U^d$ are indeed seminorms, but we will not need this either.  
\end{remark}

In \cite[Corollaries 1.26,1.27]{berg}, the following ergodic theory analogues of Theorems \ref{main}, \ref{main2} was shown:

\begin{theorem}[Inverse Conjecture for the Gowers-Host-Kra seminorm for high characteristic]\label{ergmain-thm}  Let $\X = (X,\B,\mu,(T_g)_{g \in \F^{\omega}})$ be an ergodic $\F^\omega$-system, let $\charac(\F) \geq d \geq 1$, and let $\phi \in L^\infty(\X)$ be such that $\|\phi\|_{U^d(\X)} > 0$.  Then $\|\phi\|_{u^d(\X)} > 0$.
\end{theorem} 

\begin{theorem}[Partial Inverse Conjecture for the Gowers-Host-Kra seminorm for general characteristic]\label{ergmain-thm2}  Let $\X = (X,\B,\mu,(T_g)_{g \in \F^{\omega}})$ be an ergodic $\F^\omega$-system, let $d \geq 1$, and let $\phi \in L^\infty(\X)$ be such that $\|\phi\|_{U^d(\X)} > 0$.  Then $\|\phi\|_{u^k(\X)} > 0$ for some $k=C(d)$ depending only on $d$.
\end{theorem}

\begin{remark} The ``if'' part of this theorem follows easily from van der Corput's lemma; the important part of the theorem for us is the ``only if'' part.  These results can be viewed as a finite field analogue of the results in \cite{host-kra} in high characteristic (and a partial analogue in the low characteristic case), and indeed draws heavily on the tools developed in that paper; see \cite{berg} for further discussion.  It is quite possible that $k$ can in fact be taken to equal $d$ in Theorem \ref{ergmain-thm2} (or equivalently, that the condition $\charac(\F) \geq d$ can be dropped in Theorem \ref{ergmain-thm}); this would imply Conjecture \ref{mainconj} in full generality.
\end{remark}

We will use Theorem \ref{ergmain-thm2} as a ``black box'', and it will be the primary ingredient in our proof of Theorem \ref{main2}, in much the same way that the Furstenberg recurrence theorem is the primary ingredient in Furstenberg's proof of Szemer\'edi's theorem in \cite{furst}.  Theorem \ref{ergmain-thm} plays a similar role for Theorem \ref{main}.

As with any other argument using a Furstenberg-type correspondence principle, our bounds are \emph{ineffective}, in that we do not obtain an explicit value of $\eps$ in terms of $d$ and $\delta$.  In principle, one could finitise the arguments in \cite{berg} (in the spirit of \cite{tao-quant}) to obtain such an explicit value, but this would be extremely tedious (and not entirely straightforward), and would lead to an extremely poor dependence (such as iterated tower-exponential or worse).  We will not pursue this matter here.

\subsection{Acknowledgments}

The first author is supported by a grant from the MacArthur Foundation, and by NSF grant CCF-0649473.  
The second author is supported by ISF grant 557/08,  by a Landau fellowship - supported by the Taub foundations, and by an Alon fellowship.
The authors are also greatly indebted to Ben Green for helpful conversations, and Vitaly Bergelson for encouragement.

\section{Notation}

We will rely heavily on asymptotic notation.  Given any parameters $x_1,\ldots,x_k$, we use $O_{x_1,\ldots,x_k}(X)$ to denote any quantity bounded in magnitude by $C_{x_1,\ldots,x_k} X$ for some finite quantity $C_{x_1,\ldots,x_k}$ depending only on $x_1,\ldots,x_k$.  We also write $Y \ll_{x_1,\ldots,x_k} X$ or $X \gg_{x_1,\ldots,x_k} Y$ for $Y = O_{x_1,\ldots,x_k}(X)$.  Furthermore, given an asymptotic parameter $n$ that can go to infinity, we use $o_{n \to \infty; x_1,\ldots,x_k}(X)$ to denote any quantity bounded in magnitude by $c_{x_1,\ldots,x_k}(n) X$, where $c_{x_1,\ldots,x_k}(n)$ is a quantity which goes to zero as $n \to \infty$ for fixed $x_1,\ldots,x_k$.  Thus for instance, if $r_2 > r_1 > 1$, then $\frac{\exp(r_1)}{\log r_2} = o_{r_2 \to \infty;r_1}(1)$.

\section{Statistical sampling}

It is well known that the ``global average'' $\E_{h \in V} f(h)$ of a bounded function $f: V \to \D$ can be accurately estimated (with high probability) by randomly selecting a number of points $x_1,\ldots,x_N \in V$ and computing the empirical Monte Carlo average (or ``local average'') $\E_{1 \leq n \leq N} f(x_n)$.  Indeed, it is not hard to show (by the second moment method) that with probability $o_{N \to \infty}(1)$, one has
$$ \E_{1 \leq n \leq N} f(x_n) = \E_{h \in V} f(h) + o_{N \to \infty}(1).$$
The point here is that the error term is uniform in the choice of $f$ and $V$.

We now record some variants of this standard ``random local averages approximate global averages'' fact, in which we perform more exotic empirical averages.  We begin with averages along random subspaces of $V$.

\begin{lemma}[Random sampling for integrals]\label{sampling-lem} Let $V$ be a finite-dimensional vector space, and let $f: V \to \D$ be a function.  Let $v_1,\ldots,v_m \in V$ be chosen independently at random.  Then with probability $1 - o_{m \to \infty}(1)$, we have
$$ \E_{\vec a \in \F^m} f(\vec a \cdot \vec v) = \E_{h \in V} f(h) + o_{m \to \infty}(1)$$
where $\vec v := (v_1,\ldots,v_m)$,  and $\vec a  \cdot \vec v := a_1 v_1 + \ldots + a_m v_m$.
\end{lemma}

\begin{remark} One can easily make the $o_{m \to \infty}(1)$ terms more explicit, but we will not need to do so here.
\end{remark}

\begin{proof} 
We use the second moment method. Note that
$$\E \E_{\vec a \in \F^m} f(\vec a \cdot \vec v) = \E_{h \in V} f(h) + o_{m \to \infty}(1)$$
(the $o_{m \to \infty}(1)$ error arising from the $a=0$ contribution) so by Chebyshev's inequality it suffices to show that
$$ \E | \E_{\vec a \in \F^m} f(\vec a \cdot \vec v) |^2 = |\E_{h \in V} f(h)|^2 + o_{m \to \infty}(1).$$
The left-hand side can be rearranged as
$$ \E_{\vec a, \vec b \in \F^m} 
\E f(\vec a \cdot \vec v) \bar f  (\vec b \cdot \vec v).$$
It is easy to see that the inner expectation is $|\E_{h \in V} f(h)|^2$ unless $\vec a = c\vec b$,  for some $c \in \F$ in which case it is $O(1)$.  The claim follows.
\end{proof}

In the above lemma, $f$ was deterministic and thus independent of the $v_i$.  But we can easily extend the result to the case where $f$ depends on a bounded number of the $v_i$:

\begin{corollary}[Random sampling for integrals, II]\label{sampling-2} Let $V$ be a finite-dimensional vector space, let $m \geq m_0 \geq 0$, let $v_1,\ldots,v_m \in V$ be chosen independently at random, and let $f_{v_1,\ldots,v_{m_0}}: V \to \D$ be a function that depends on $v_1,\ldots,v_{m_0}$ but is independent of $v_{m_0+1},\ldots,v_m$.
Then with probability $1 - o_{m \to \infty; m_0}(1)$, we have
$$ \E_{\vec a \in \F^m} f_{v_1,\ldots,v_{m_0}}(\vec a \cdot \vec v) = \E_{h \in V} f_{v_1,\ldots,v_{m_0}}(h) + o_{m \to \infty; m_0}(1).$$
\end{corollary}

\begin{proof}  We write $\vec a = (\vec a_0, \vec a_1) \in \F^{m_0} \times \F^{m-m_0}$ and $\vec v = (\vec v_0, \vec v_1) \in V^{m_0} \times V^{m-m_0}$.  If we condition $\vec v_0 = (v_1,\ldots,v_{m_0})$ to be fixed, we see from applying Lemma \ref{sampling-lem} to the remaining random vectors $\vec v_1$ that for fixed $\vec a_0$, we have
$$ \E_{\vec a_1 \in \F^{m-m_0}} f_{v_1,\ldots,v_{m_0}}(\vec a \cdot \vec v) = \E_{h \in V} f_{v_1,\ldots,v_{m_0}}(\vec a_0 \cdot \vec v_0 + h) + o_{m-m_0 \to \infty}(1)$$
with probability $1 - o_{m -m_0 \to \infty}(1)$ conditioning on $\vec v_0$; integrating this we see that the same is true without the conditioning.  We can shift $h$ by $\vec a_0 \cdot \vec v_0$, move the $h$ average onto the other side, and take expectations to conclude that
$$ \E |\E_{\vec a_1 \in \F^{m-m_0}} f_{v_1,\ldots,v_{m_0}}(\vec a \cdot \vec v) - \E_{h \in V} f_{v_1,\ldots,v_{m_0}}(h)| = o_{m-m_0 \to \infty}(1)$$
for each $\vec a_0$; averaging over $\vec a_0$ by the triangle inequality we obtain the claim.
\end{proof}

\begin{remark}\label{cyclic} It is with this corollary that we are implicitly exploiting the highly transitive nature of the symmetry group $GL(V)$ available to us.  In the setting of the cyclic group $\Z/N\Z$, the analogue of Lemma \ref{sampling-lem} is still true, namely that one can approximate a global average $\int_{\Z/N\Z} f$ by a local average on random arithmetic progressions of medium length, but this approximation no longer holds if $f$ is allowed to depend on the first few values of that progression, since this of course determines the rest of the progression; this is related to the fact that (for $N$ prime, say), the affine group of $\Z/N\Z$ (which is analogous to $GL(V)$) is $2$-transitive but no stronger.  In contrast, in the finite field setting, a small subspace of a medium-dimensional subspace does not determine the whole subspace.
\end{remark}

We will need to generalise these results further by considering more exotic averages along cubes.  A typical result we will need can be stated informally as
\begin{equation}\label{cute}
\begin{split}
&\E_{\vec a_2 \in \F^{m_2}} \E_{\vec a_1 \in \F^{m_1}} 
\int_V f (T_{\vec a_1 \cdot \vec v_1} \bar{f}) (T_{\vec a_2 \cdot \vec v_2} \bar{f}) (T_{\vec a_1 \cdot \vec v_1 + \vec a_2 \cdot \vec v_2} f) \\
&\approx \E_{h_1, h_2 \in V} \int_V f (T_{h_1} \bar{f}) (T_{h_2} \bar{f}) T_{h_1+h_2} f 
\end{split}
\end{equation}
when $m_1$ is large, $m_2$ is large compared with $m_1$, and $\vec v$ is random (see Lemma \ref{gow-lem} for the formal version of this type of estimate).  Such results follow (heuristically, at least), by iterating the previous results.  For instance, from Corollary \ref{sampling-2} we heuristically have
\begin{align*}
&\E_{\vec a_2 \in \F^{m_2}} \E_{\vec a_1 \in \F^{m_1}} 
\int_V f (T_{\vec a_1 \cdot \vec v_1} \bar{f}) (T_{\vec a_2 \cdot \vec v_2} \bar{f}) (T_{\vec a_1 \cdot \vec v_1 + \vec a_2 \cdot \vec v_2} f)(x) \\
&\quad \approx \E_{h_2 \in V} \E_{\vec a_1 \in \F^{m_1}} 
\int_V f (T_{\vec a_1 \cdot \vec v_1} \bar{f}) (T_{h_2} \bar{f}) (T_{\vec a_1 \cdot \vec v_1 + h_2} f)
\end{align*}
when $m_2$ is large compared to $m_1$ and then interchanging the expectations and applying Lemma \ref{sampling-lem} heuristically yields
\begin{align*}
&\E_{h_2 \in V} \E_{\vec a_1 \in \F^{m_1}} 
\int_V f (T_{\vec a_1 \cdot \vec v_1} \bar{f}) (T_{h_2} \bar{f}) (T_{\vec a_1 \cdot \vec v_1 + h_2} f) \\
&\quad \approx \E_{h_1 \in V} \E_{h_2 \in V} 
\int_V f (T_{h_1} \bar{f}) (T_{h_2} \bar{f}) (T_{h_1 + h_2} f) 
\end{align*}
when $m_1$ is large, thus giving \eqref{cute}.  

We will formalise the precise statement along these lines that we need later in this section.  We begin with some key definitions.

\begin{definition}[Lipschitz norm]  If $G: \D^n \to \C$ is a function on a polydisk $\D^n$, we define the \emph{Lipschitz norm} $\|G\|_{\Lip}$ of $G$ to be the quantity
$$ \|G\|_{\Lip} := \sup_{z \in \D^n} |G(z)| + \sup_{z,w \in \D^n: z \neq w} \frac{|G(z)-G(w)|}{d(z,w)}$$
where we use the metric
$$ d( (z_1,\ldots,z_n), (w_1,\ldots,w_n) ) := |z_1-w_1| + \ldots + |z_n-w_n|.$$
\end{definition}

\begin{definition}[Accurate sampling sequence]\label{ass}  Let $k \geq 1$, let $V$ be a finite-dimensional vector space, let $f: V \to \D$ be a bounded function, and let
$$ 0= H_0 < H_1 < H_2 < H_3 < \ldots$$
be a sequence of integers (or ``scales'').  We define an \emph{accurate sampling sequence} for $f$ of degree $k$ and at scales $H_1,H_2,\ldots$ to be an infinite sequence of vectors
$$ v_1, v_2, v_3, \ldots \in V$$
such that for every sequence
$$ 0 \leq r_0 < r_1 < \ldots < r_k$$
of scales and every Lipschitz function $G: \D^{\{0,1\}^k \times \F^{H_{r_0}}} \to \C$, we have
\begin{equation}\label{intv}
 \int_V | G_{f,r_0,r_1,\ldots,r_k} - G_{f,r_0} | \leq \frac{\|G\|_{\Lip}}{r_1} 
\end{equation}
where 
$$ G_{f,r_0,r_1,\ldots,r_k}(x) := \E_{\vec a_1 \in \F^{H_{r_1}}, \ldots, \vec a_k \in \F^{H_{r_k}}} 
G( (f(x+{\bf \omega } \cdot {\bf u} + \vec b \cdot \vec v_0))_{{\bf \omega} \in \{0,1\}^k, \vec b \in \F^{H_{r_0}}} ),$$
where 
\[
{\bf u} = (\vec a_1 \cdot \vec v_1,\ldots ,\vec a_k \cdot \vec v_k) ;  \quad \vec v_j=(v_1,\ldots, v_{H_{r_j}}), \ j=0,\ldots, k,
\]
and 
$$ G_{f,r_0}(x) := \E_{h_1 \in V, \ldots, h_k \in V} 
G( (f(x+{\bf \omega}\cdot  {\bf h}  + \vec b \cdot \vec v_0))_{{\bf \omega} \in \{0,1\}^k, \vec b \in \F^{H_{r_0}}} ),$$
where ${\bf h} =(h_1,\ldots,h_k) $.
\end{definition}

\begin{remark}  The denominator $r_1$ in \eqref{intv} could be replaced by any other fixed function of $r_1$ that went to infinity as $r_1 \to \infty$ if desired here.
\end{remark}

\begin{remark}\label{descent}
We make the trivial but useful remark that an accurate sampling sequence of degree $k$ is also an accurate sampling sequence of degree $k'$ for any $1 \leq k' \leq k$.  Indeed, to verify \eqref{intv} for a function $G': \D^{\{0,1\}^{k'} \times \F^{H_{r_0}}} \to \D$ and some scales $r_{k'} > \ldots > r_0 \geq 0$, one simply adds some dummy scales $r_{k'+1},\ldots,r_k$ above $r_{k'}$ and extends $G'$ to a function $G: \D^{\{0,1\}^k \times \F^{H_{r_0}}} \to \D$ by composing with the obvious restriction map from $\D^{\{0,1\}^k \times \F^{H_{r_0}}}$ to $\D^{\{0,1\}^{k'} \times \F^{H_{r_0}}}$.
\end{remark}

Roughly speaking, an accurate sampling sequence will allow us to estimate all the global averages that we need for the combinatorial inverse conjecture for the Gowers norm by local averages which are suitable for lifting to the ergodic setting via the correspondence principle.  We illustrate the use of such sequences by describing the three special cases of \eqref{intv} that we will actually need in our arguments.

\begin{lemma}[Global Gowers norm can be approximated by local Gowers norm]\label{gow-lem} 
Let $d \geq 1$, let $V$ be a finite-dimensional vector space, let $f: V \to \D$ be a bounded function, and let $v_1, v_2, \ldots \in V$ be an accurate sampling sequence for $f$ of degree $d$ and at scales $H_1, H_2, \ldots$.  Then for every sequence of scales
$$ 0 < r_1 < r_2 < \ldots < r_d$$
we have
$$ \E_{\vec a_1 \in \F^{H_{r_1}}, \ldots, \vec a_d \in \F^{H_{r_d}}} \int_V \mder_{\vec a_1 \cdot \vec v_{r_1}} \ldots \mder_{\vec a_d \cdot \vec v_{r_d}} f = \| f \|_{U^d(V)}^{2^d} + o_{r_1 \to \infty; d}(1).$$
\end{lemma}

\begin{remark}
As with all other estimates in this section, the point is that the error term is uniform over all choices of $f$ and $V$.  Note that the $d=2$ case of this lemma is a formalisation of \eqref{cute}.
\end{remark}

\begin{proof} We apply \eqref{intv} with $r_0=0$, and $G: \D^{\{0,1\}^d} \to \C$ being the function
$$ G\left( (z({\bf \omega}))_{{\bf \omega} \in \{0,1\}^d} \right) := \prod_{{\bf \omega} \in \{0,1\}^d} {\mathcal C}^{\omega_1 + \ldots + \omega_d} z({\bf \omega})$$
where ${\mathcal C}: z \mapsto \overline{z}$ is the complex conjugation operator.
A routine computation gives the identities
$$ G_{f,0,r_1,\ldots,r_d}(x) = \E_{\vec a_1 \in \F^{H_{r_1}}, \ldots, \vec a_d \in \F^{H_{r_d}}} \mder_{\vec a_1 \cdot \vec v_{r_1}} \ldots \mder_{\vec a_d \cdot \vec v_{r_d}} f$$
and
$$ \int_V G_{f,0} = \| f \|_{U^d(V)}^{2^d}.$$
Also, it is easy to see that the Lipschitz norm $\|G\|_{\Lip}$ is $O_d(1)$.  The claim now follows immediately from \eqref{intv} and the triangle inequality.
\end{proof}

\begin{lemma}[Global averages can be approximated by local averages]\label{avg-lem} 
Let $V$ be a finite-dimensional vector space, let $f: V \to \D$ be a bounded function, and let $v_1, v_2, \ldots \in V$ be an accurate sampling sequence for $f$ of degree $1$ and at scales $H_1, H_2, \ldots$.  Then for every finite sequence $\vec b_1,\ldots,\vec b_m \in \F^\omega$ and
every continuous function $F: \D^m \to \C$, we have
$$ \int_V |\E_{\vec a \in \F^{H_r}} T_{\vec a \cdot \vec v} g - \int_V g| = o_{r \to \infty; F, m, \vec b_1,\ldots,\vec b_m}(1)$$
where $g: V \to \C$ is the function
\begin{equation}\label{gdef-eq} g(x) := F( T_{\vec b_1 \cdot \vec v} f(x), \ldots, T_{\vec b_m \cdot \vec v} f(x) ).
\end{equation}
\end{lemma}

\begin{proof} By approximating the continuous function $F$ uniformly by a Lipschitz function, we may assume that $F$ is Lipschitz.  By adding dummy vectors to the collection $\vec b_1,\ldots,\vec b_m$ if necessary, we may assume that $\{\vec b_1,\ldots,\vec b_m\} = \F^{H_{r_0}}$ for some $r_0 > 0$ depending on $\vec b_1,\ldots,\vec b_m$, thus $F$ is now a Lipschitz function from $\D^{\F^{H_{r_0}}}$ to $\C$.  

Note that to prove the claim we may without loss of generality restrict to the regime $r > r_0$.
We now apply \eqref{intv} with $G: \D^{\{0,1\} \times \F^{H_{r_0}}} \to \C$ being the function
$$ G\left( (z(\omega,\vec b))_{\omega \in \{0,1\}, \vec b \in \F^{H_{r_0}}} \right) := F\left( (z(1,\vec b))_{\vec b \in \F^{H_{r_0}}} \right).$$
A routine computation gives the identities
$$ G_{f,r_0,r}(x) = \E_{\vec a \in \F^{H_r}} T_{\vec a \cdot \vec v} g(x)$$
and
$$ G_{f,r_0}(x) = \E_{h \in V} T_h g(x) = \int_V g.$$
Also, it is clear that $G$ is Lipschitz with norm $O_{F, r_0}(1)$.  The claim then follows from \eqref{intv}.
\end{proof}

\begin{lemma}[Global polynomiality test can be approximated by local polynomiality test]\label{test-lem} 
Let $k \geq 1$, let $V$ be a finite-dimensional vector space, let $f: V \to \D$ be a bounded function, and let $v_1, v_2, \ldots \in V$ be an accurate sampling sequence for $f$ of degree $k$ and at scales $H_1, H_2, \ldots$.  Then for every finite sequence $\vec b_1,\ldots,\vec b_m \in \F^\omega$ and every continuous function $F: \D^m \to \C$, we have
\begin{align*}
&\E_{\vec a_1 \in \F^{H_{r_1}}} \ldots \E_{\vec a_k \in \F^{H_{r_k}}} 
\int_V |\mder_{\vec a_1 \cdot \vec v} \ldots \mder_{\vec a_k \cdot \vec v} g - 1| \\
&\quad =
\E_{h_1,\ldots,h_k \in V}
\int_V |\mder_{h_1} \ldots \mder_{h_k} g - 1| + o_{r_1 \to 0; F, m, \vec b_1,\ldots, \vec b_m, k}(1)
\end{align*}
for any $1 \leq r_1 < r_2 < \ldots < r_k$, where $g: V \to \C$ is the function defined by \eqref{gdef-eq}.
\end{lemma}

\begin{proof} Arguing as in Lemma \ref{avg-lem}, we may assume that $\{\vec b_1,\ldots,\vec b_m \} = \F^{H_{r_0}}$ for some $r_0 > 0$ depending on $\vec b_1,\ldots,\vec b_m$, and that $F: \D^{\F^{H_{r_0}}} \to \C$ is Lipschitz.

Note that to prove the claim we may without loss of generality restrict to the regime $r_1 > r_0$.
We now apply \eqref{intv} with $G: \D^{\{0,1\}^k \times \F^{H_{r_0}}} \to \C$ being the function
$$ G\left( (z({\bf \omega},\vec b))_{{\bf \omega} \in \{0,1\}^k, \vec b \in \F^{H_{r_0}}} \right) := 
\left|\prod_{{\bf \omega} \in \{0,1\}^k} {\mathcal C}^{\omega_1+\ldots+\omega_k}
F\left( (z({\bf \omega},\vec b))_{\vec b \in \F^{H_{r_0}}} \right) - 1\right|$$
where ${\mathcal C}$ is again the complex conjugation operator.
A routine computation gives the identities
$$ G_{f,r_0,r_1,\ldots,r_k}(x) = 
\E_{\vec a_1 \in \F^{H_{r_1}}} \ldots \E_{\vec a_k \in \F^{H_{r_k}}} |\mder_{\vec a_1 \cdot \vec v} \ldots \mder_{\vec a_k \cdot \vec v} g(x) - 1|$$
and
$$ G_{f,r_0}(x) = \E_{h_1,\ldots,h_k \in V} |\mder_{h_1} \ldots \mder_{h_k} g(x) - 1|$$
for any $r_0 < r_1 < \ldots < r_k$.  Also it is clear that $G$ is Lipschitz with norm $O_{F, r_0, k}(1)$.  The claim then follows from \eqref{intv} and the triangle inequality.
\end{proof}

Of course, in order to utilise the above lemmas we need to know that such accurate sampling sequences in fact exist.  This is the purpose of the following proposition.

\begin{proposition}[Existence of accurate sampling sequence]\label{hhh}  Let $d \geq 1$.  Then there exists a sequence
$$ 0=H_0 < H_1 < H_2 < H_3 < \ldots$$
of integers such that for every finite-dimensional vector space $V$ and any function $f: V \to \D$, there exists an accurate sampling sequence $v_1, v_2, v_3, \ldots \in V$ for $f$ of degree $d$ at scales $H_1,H_2,\ldots$.
\end{proposition}

\begin{remark} The key point here is that the scales $H_1, H_2, H_3, \ldots$ are \emph{universal}; they depend on $d$, but otherwise and work for all vector spaces $V$ and functions $f$.
\end{remark}

\begin{proof}  We select $H_j$ recursively by the formula $H_{j+1} := F(H_j)$, where $F = F_d: \N \to \N$ is a sufficiently rapidly growing function depending on $d$ that we will choose later.

We use the probabilistic method, choosing $v_1,v_2,\ldots \in V$ uniformly at random, and showing that (if $F$ was sufficiently rapid) the resulting sequence will be an accurate sampling sequence with positive probability.

We begin with observing that in order to verify the condition \eqref{intv}, it suffices by the triangle inequality to show that with positive probability, one has
\begin{equation}\label{intv-d}
 \int_V | G_{f,r_0,r_1,\ldots,r_{d'}} - G_{f,r_0,r_1,\ldots,r_{d'-1}} | \leq \frac{\|G\|_{\Lip}}{dr_1} 
\end{equation}
for all $1 \leq d' \leq d$, all $0 \leq r_0 < \ldots < r_{d'}$, and every Lipschitz function $G: \D^{\{0,1\}^d \times \F^{H_{r_0}}} \to \C$, where
\begin{align*}
 G_{f,r_0,r_1,\ldots,r_{d'}}(x) &:= \E_{\vec a_1 \in \F^{H_{r_1}}, \ldots, \vec a_{d'} \in \F^{H_{r_{d'}}}} \E_{h_{d'+1},\ldots,h_d \in V}\\
&\quad G\left( (f(x+\sum_{j=1}^{d'} \omega_j \vec a_j \cdot \vec v_j + \sum_{j=d'+1}^d \omega_j h_j + b \cdot \vec v_0))_{(\omega_1,\ldots,\omega_d) \in \{0,1\}^d, b \in \F^{H_{r_0}}} \right).
\end{align*}

By the union bound, it will suffice to show that for all $1 \leq d' \leq d$ and all $0 \leq r_0 < \ldots < r_{d'}$, with probability $1 - o_{H_{r_{d'}} \to \infty; d, H_{r_0},\ldots,H_{r_{d'-1}},r_1}(1)$, \eqref{intv-d} holds for all Lipschitz functions $G: \D^{\{0,1\}^d \times \F^{H_{r_0}}} \to \C$, since the total failure probability can be made to be less than $1$ by choosing $F$ to be sufficiently rapid.

We can normalise $G$ to have Lipschitz norm $1$.  By the Arzel\'a-Ascoli theorem, the space of such functions is compact in the uniform topology.  In particular, there exists a collection of functions $G: \D^{\{0,1\}^d \times \F^{H_{r_0}}} \to \C$ of Lipschitz norm $1$, $S$, of size $O_{d,H_{r_0},r_1}(1)$, such that any other such Lipschitz function lies within $\frac{1}{4dr_1}$ (say) of a function $G \in  S$ in the uniform metric.  Because of this, we see from the union bound again that it will suffice to show that for all $1 \leq d' \leq d$ and all $0 \leq r_0 < \ldots < r_{d'}$, and all functions $G: \D^{\{0,1\}^d \times \F^{H_{r_0}}} \to \C$ of Lipschitz norm $1$ in $S$,
\begin{equation}\label{intv-2}
 \int_V | G_{f,r_0,r_1,\ldots,r_{d'}} - G_{f,r_0,r_1,\ldots,r_{d'-1}} | \leq \frac{1}{2dr_1} 
\end{equation}
of \eqref{intv-d} holds with probability $1 - o_{H_{r_{d'}} \to \infty; d, H_{r_0},\ldots,H_{r_{d'-1}},r_1}(1)$.

Fix $d', r_0,\ldots,r_{d'},G$.  By Markov's inequality, it suffices to show that
$$
\E  \int_V | G_{f,r_0,r_1,\ldots,r_{d'}} - G_{f,r_0,r_1,\ldots,r_{d'-1}} | 
= o_{H_{r_{d'}} \to \infty; d, H_{r_0},\ldots,H_{r_{d'-1}}}(1);$$
by linearity of expectation it thus suffices to show that
$$
\E  | G_{f,r_0,r_1,\ldots,r_{d'}}(x) - G_{f,r_0,r_1,\ldots,r_{d'-1}}(x) | 
= o_{H_{r_{d'}} \to \infty; d, H_{r_0},\ldots,H_{r_{d'-1}}}(1)$$
uniformly in $x \in V$.

Fix $x$.  We observe that
$$ G_{f,r_0,r_1,\ldots,r_{d'}}(x) = \E_{\vec a \in \F^{H_{r_{d'}}}}  f_{v_1,\ldots,v_{H_{r_{d'-1}}}}( \vec a \cdot \vec v_{d'} )$$
and
$$ G_{f,r_0,r_1,\ldots,r_{d'-1}}(x) = \E_{h \in V}  f_{v_1,\ldots,v_{H_{r_{d'-1}}}}( h )$$
where $f_{v_1,\ldots,v_{H_{r_{d'-1}}}}: V \to \D$ is the function
\begin{align*}
&f_{v_1,\ldots,v_{H_{r_{d'-1}}}}( h ) :=
\E_{\vec a_1 \in \F^{H_{r_1}}, \ldots, \vec a_{d'-1} \in \F^{H_{r_{d'-1}}}} \E_{h_{d'+1},\ldots,h_d \in V}\\
&G\left( (f(x+\sum_{j=1}^{d'-1} \omega_j \vec a_j \cdot \vec v_j + \omega_{d'} h_{d'} + \sum_{j=d'+1}^d \omega_j h_j + b \cdot \vec v_0))_{(\omega_1,\ldots,\omega_d) \in \{0,1\}^d, b \in \F^{H_{r_0}}} \right).
\end{align*}
As the notation suggests, the function $f_{v_1,\ldots,v_{H_{r_{d'-1}}}}$ depends on the values of $v_1,\ldots,v_{H_{r_{d'-1}}}$ but not on higher elements of the sequence.  Also, as $G$ has Lipschitz norm $1$, $f$ takes values in $\D$.  The claim now follows from Corollary \ref{sampling-2}.
\end{proof}

\section{Proof of main theorems}

We are now ready to prove the main theorems.  We shall just prove Theorem \ref{main2} using Theorem \ref{ergmain-thm2}; the deduction of Theorem \ref{main} using Theorem \ref{ergmain-thm} is exactly analogous (see the brief remarks at the end of this section).

Fix $\F$ and $d$, and let $k = C(d)$ be the quantity in Theorem \ref{ergmain-thm2}.  By increasing $k$ if necessary we may assume $k \geq d$.  Assume for sake of contradiction that Theorem \ref{main2} failed for this choice of $\F, d, k$.  Then we can find $\delta > 0$ and a sequence $f^{(n)}: V^{(n)} \to \D$ of functions on finite-dimensional vector spaces $V^{(n)}$ such that
\begin{equation}\label{faun}
\|f^{(n)}\|_{U^d(V^{(n)})} \geq \delta
\end{equation}
for all $n$, but
\begin{equation}\label{fund}
\|f^{(n)}\|_{u^k(V^{(n)})} = o_{n \to \infty}(1).
\end{equation}
We now let $F(x) := x$, and let 
$$ 1 < H_1 < H_2 < \ldots$$
be the sequence in Proposition \ref{hhh}; it is important to note that this sequence does not depend on $n$.  From that proposition, we can find an accurate sampling sequence
$$ v^{(n)}_1, v^{(n)}_2, \ldots \in V^{(n)}$$
for $f^{(n)}$ of degree $k$ at these scales.  We fix such a sequence for each $n$.

We will use these sampling sequences to lift the functions $f^{(n)}$ on $V^{(n)}$ to a universal dynamical system for $\F^\omega$ by the usual Furstenberg correspondence principle method.  We begin by constructing this universal space.

\begin{definition}[Furstenberg universal space]
Let $X := \D^{\F^\omega}$ be the space of functions $\zeta: \F^\omega \to \D$.  With the product topology, this is a compact metrisable space, with Borel $\sigma$-algebra $\B$.  It has a continuous action $h \mapsto T_h$ of the additive group $\F^\omega$, defined by the formula
$$ T_h \zeta(x) := \zeta(x+h).$$
We let $\Pr(X)^T$ be the space of all Borel probability measures $\mu$ on $X$ which are invariant with respect to this action; note that $\X = (X,\B,\mu,(T_h)_{h \in \F^\omega})$ is a $\F^\omega$-system for any $\mu \in \Pr(X)^T$.  If $\mu^{(n)} \in \Pr(X)^T$ is a sequence of such measures, and $\mu \in \Pr(X)^T$ is another measure, we say that $\mu^{(n)}$ \emph{converges vaguely} to $\mu$ if we have
$$ \lim_{n \to \infty} \int_X \phi(\zeta)\ d\mu^{(n)}(\zeta) \to \int_X \phi(\zeta)\ d\mu(\zeta)$$
for all continuous functions $\phi: X \to \C$.
\end{definition}

Because $X$ is compact metrisable, and the action of $T$ is continuous it is a well-known fact that $\Pr(X)^T$ is sequentially compact; thus every sequence of measures in $\Pr(X)^T$ has a vaguely convergent subsequence whose limit is also in $\Pr(X)^T$.

For each $n$, we define a measure $\mu^{(n)} \in \Pr(X)^T$ on $X$ by the formula
$$ \mu^{(n)} = \E_{x \in V^{(n)}} \delta_{\zeta_{n,x}}$$
where $\delta$ denotes the Dirac mass and for each $x \in V^{(n)}$, $\zeta_{n,x} \in X$ is the function
$$ \zeta_{n,x}(\vec a) := T_{\vec a \cdot \vec v^{(n)}}f^{(n)}(x)=T_{\sum_{m=1}^\infty a_m v^{(n)}_m} f^{(n)}(x)$$
for all $\vec a  \in \F^\omega$ (note the sum on the right-hand side has only finitely many non-zero terms).
Observe that $\mu^{(n)}$ is indeed $T$-invariant.
By passing to a subsequence if necessary, we may assume that $\mu^{(n)}$ converges vaguely to a limit $\mu \in \Pr(X)^T$.  We write $\X := (X, \B, \mu, (T_h)_{h \in \F^\omega})$.

Let $f: X \to \D$ be the indicator function $f(\zeta) := \zeta(0)$.  We observe the key correspondence
\begin{equation}\label{ident-eq}
\int_X G(T_{\vec a_1} f, \ldots, T_{\vec a_k} f)\ d\mu^{(n)}(\zeta) = \int_{V^{(n)}} G( T_{\vec a_1 \cdot \vec v^{(n)}} f^{(n)}, \ldots, T_{\vec a_k \cdot \vec v^{(n)}} f^{(n)} )
\end{equation}
for all $\vec a_1,\ldots,\vec a_k \in \F^\omega$, all $n$, and all continuous $G: \D^k \to \C$.

We now record the (standard) fact that the countable collection of shifts $T_h f$ for $h \in \F^\omega$ generate $L^\infty(\X)$:

\begin{lemma}[$T_h f$ generate $L^\infty(\X)$]\label{Gen}  Given any $\phi \in L^\infty(\X)$ and $\eps > 0$, there exists a finite number of shifts $\vec h_1,\ldots,\vec h_k \in \F^\omega$ and a continuous function $G: \D^k \to \C$ such that $$\int_X |\phi - G( T_{\vec h_1} f, \ldots, T_{\vec h_k} f )|\ d\mu \leq \eps.$$
\end{lemma}

\begin{proof} For continuous $\phi$, the claim follows easily from the Stone-Weierstrass theorem (and in this case we can upgrade the $L^1$ approximation to $L^\infty$ approximation).  As $X$ is compact metrisable, the Borel measure $\mu$ is in fact a Radon measure, and so (by Urysohn's lemma) the continuous functions are dense in $L^\infty(\X)$ in the $L^1(\X)$ topology, and the claim follows.
\end{proof}

We can now use the machinery of the previous section to deduce various important facts about $\X$ and $f$.  For instance, Lemma \ref{avg-lem} now implies

\begin{lemma}[Ergodicity] $\X$ is ergodic.
\end{lemma}

\begin{proof}
By the mean ergodic theorem, it suffices to show that
$$ \lim_{r \to \infty} \int_X \left| \E_{\vec h \in \F^{H_r}} T_h g - \int_X g\ d\mu\right|\ d\mu = 0$$
for all $g \in L^\infty(X)$.  By Lemma \ref{Gen} and a standard limiting argument it suffices to show this for $g$ which are functions of finitely many shifts of $f$, say $g = G( T_{\vec b_1} f, \ldots, T_{\vec b_k} f )$.  We will then show that
$$ \int_X \left| \E_{\vec h \in \F^{H_r}} T_{\vec h} g - \int_X g\ d\mu \right|\ d\mu = o_{r \to \infty; G, k, \vec b_1,\ldots,\vec b_k}(1).$$
By vague convergence it suffices to show that
$$ \int_X \left| \E_{h \in \F^{H_r}} T_{\vec h} g - \int_X g\ d\mu^{(n)}\right|\ d\mu^{(n)} = o_{r \to \infty; G, k, \vec b_1,\ldots,\vec b_k}(1)$$
for all $n$.  By \eqref{ident-eq}, we can rewrite the left-hand side as
$$ \int_V \left| \E_{\vec h \in \F^{H_r}} T_{\vec h \cdot \vec v^{(n)}_r} g^{(n)} - \int_V g^{(n)}\right|$$
where 
$$g^{(n)} := G( T_{\vec b_1 \cdot \vec v^{(n)}} f^{(n)}, \ldots, T_{\vec b_k \cdot \vec v^{(n)}} f^{(n)}).$$
But the claim now follows from Lemma \ref{avg-lem} (and Remark \ref{descent}).
\end{proof}

In a similar spirit, Lemma \ref{gow-lem} implies

\begin{lemma}[$f$ has large Gowers-Host-Kra norm]  We have $\|f\|_{U^d(\X)} \geq \delta$.
\end{lemma}

\begin{proof}  From the mean ergodic theorem we have
$$ \| f\|_{U^1(\X)}^2 = \limsup_{K_1 \to \infty} \E_{\vec h_1 \in \F^{K_1}} \int_X \mder_{\vec h_1} f \ d\mu$$
and by induction we have
$$ \| f\|_{U^d(\X)}^{2^d} = \limsup_{K_d \to \infty} \ldots \limsup_{K_1 \to \infty} 
\E_{\vec h_d \in \F^{K_d}} \ldots \E_{\vec h_1 \in \F^{K_1}} \int_X \mder_{\vec h_1} \ldots \mder_{\vec h_d} f \ d\mu.$$
It thus suffices to show that 
$$ \E_{\vec h_d \in \F^{H_{r_d}}} \ldots \E_{\vec h_1 \in \F^{H_{r_1}}} \int_X \mder_{\vec h_1} \ldots \mder_{\vec h_d} f \ d\mu > \delta^{2^d} - o_{r_d \to \infty}(1)$$
whenever
$$ 1 \leq r_d < \ldots < r_1.$$
By reversing the order of averages, it suffices to show that
$$ \E_{\vec h_d \in \F^{H_{r_d}}} \ldots \E_{\vec h_1 \in \F^{H_{r_1}}} \int_X \mder_{\vec h_1} \ldots \mder_{\vec  h_d} f \ d\mu > \delta^{2^d} - o_{r_1 \to \infty}(1)$$
whenever
$$ 1 \leq r_1 < \ldots < r_d.$$
Fix $r_1,\ldots,r_d$.  By weak convergence, it suffices to show that
$$ \E_{\vec h_d \in \F^{H_{r_d}}} \ldots \E_{\vec  h_1 \in \F^{H_{r_1}}} \int_X \mder_{\vec h_1} \ldots \mder_{\vec h_d} f \ d\mu^{(n)} > \delta^{2^d} - o_{r_1 \to \infty}(1)$$
for all $n$.  By \eqref{faun}, it suffices to show that
$$ \E_{\vec h_d \in \F^{H_{r_d}}} \ldots \E_{\vec h_1 \in \F^{H_{r_1}}} \int_X \mder_{\vec h_1} \ldots \mder_{\vec h_d} f \ d\mu^{(n)} > \| f^{(n)} \|_{U^d(V^{(n)})}^{2^d} - o_{r_1 \to \infty}(1).$$
By \eqref{ident-eq}, left-hand side can be rephrased as
$$ \int_V \E_{\vec a_1 \in \F^{H_{r_1}}, \ldots, \vec a_d \in \F^{H_{r_d}}} \mder_{\vec a_1 \cdot \vec v_{r_1}^{(n)}} \ldots \mder_{\vec a_d \cdot \vec v_{r_d}^{(n)}} f^{(n)}$$
and the claim now follows from Lemma \ref{gow-lem} (and Remark \ref{descent}).
\end{proof}

We have now verified all the hypotheses of Theorem \ref{ergmain-thm}.  Applying that theorem, we conclude that $\|f\|_{u^k(\X)} > c$ for some $c > 0$ (which could be very small, but positive).  Thus we can find a phase polynomial $\phi \in \Phase_{k-1}(\X)$ of degree $k-1$ such that
$$ |\int_X f \overline{\phi}\ d\mu| > c.$$
Let $\eps > 0$ be a small number (depending on $d, k, c$) to be chosen later.  By Lemma \ref{Gen}, we can find $\vec b_1,\ldots,\vec b_m \in \F^\omega$ (with $m$ potentially quite large, but finite) and a continuous $G: \D^m \to \C$ such that
\begin{equation}\label{slither-eq} \int_X |\phi - G( T_{\vec b_1} f, \ldots, T_{\vec b_m} f )| \leq \eps.
\end{equation}
Since $\phi$ takes values in $\D$, we may assume without loss of generality that $G$ does also.
If $\eps$ is small enough depending on $c$, we thus have
$$ |\int_X f \overline{G(T_{\vec b_1} f, \ldots, T_{\vec b_m} f)}\ d\mu| > c/2.$$
By vague convergence, we thus have
$$ |\int_X f \overline{G(T_{\vec b_1} f, \ldots, T_{\vec b_m} f)}\ d\mu^{(n)}| > c/4$$
for all sufficiently large $n$ (depending on $G, m, c$).  Using \eqref{ident-eq}, we rearrange this as
\begin{equation}\label{c4}
|\int_V f^{(n)} \overline{G(T_{\vec b_1 \cdot \vec v^{(n)}} f^{(n)}, \ldots, T_{\vec b_m \cdot \vec v^{(n)}} f^{(n)})}| > c/4.
\end{equation}
Now let $r_1$ be a large integer depending on the $\vec b_1,\ldots,\vec b_m,\eps$, and let $r_j := r_1 + (j-1)$ for $j=2,\ldots,d$.  Since $\phi$ is a phase polynomial of degree $k-1$, we have
$$ \int_X |\mder_{\vec a_1} \ldots \mder_{\vec a_k} \phi - 1|\ d\mu = 0$$
for all $\vec a_1 \in \F^{H_{r_1}}, \ldots,\vec a_k \in \F^{H_{r_k}}$.  From many applications of \eqref{slither-eq}, the triangle inequality, and the boundedness of $\phi, G$, we conclude that
$$ \int_X |\mder_{\vec a_1} \ldots \mder_{\vec a_k} G(T_{\vec b_1} f, \ldots, T_{\vec b_m} f) - 1|\ d\mu \ll_k \eps$$
for all $\vec a_1 \in \F^{H_{r_1}}, \ldots,\vec a_k \in \F^{H_{r_k}}$.  By vague convergence, this implies that
$$ \int_X |\mder_{\vec a_1} \ldots \mder_{\vec a_k} G(T_{\vec b_1} f, \ldots, T_{\vec b_m} f) - 1|\ d\mu^{(n)} \ll_k \eps$$
for all sufficiently large $n$ (depending on $\eps, H_{r_1},\ldots,H_{r_k}$).  Using \eqref{ident-eq}, we can rearrange the left-hand side as
$$ \int_{V^{(n)}} |\mder_{\vec a_1 \cdot \vec v^{(n)}} \ldots \mder_{\vec a_k \cdot \vec v^{(n)}} G(T_{\vec b_1\cdot \vec v^{(n)}} f^{(n)}, \ldots, T_{\vec b_m\cdot \vec v^{(n)}} f^{(n)}) - 1|$$
and so on averaging we obtain
$$ \E_{\vec a_1 \in \F^{H_{r_1}},\ldots,\vec a_k \in \F^{H_{r_k}}} \int_{V^{(n)}} |\mder_{\vec a_1 \cdot \vec v^{(n)}} \ldots \mder_{\vec a_k \cdot \vec v^{(n)}} G(T_{\vec b_1\cdot \vec v^{(n)}} f^{(n)}, \ldots, T_{\vec b_m\cdot \vec v^{(n)}} f^{(n)}) - 1| \ll_k \eps.$$
Applying Lemma \ref{test-lem} we conclude (if $r_1$ is sufficiently large depending on $\vec b_1,\ldots,\vec b_m,\eps$) that
$$ \E_{h_1,\ldots,h_k \in V^{(n)}} \int_{V^{(n)}} |\mder_{ h_1} \ldots \mder_{h_k} G(T_{\vec b_1\cdot \vec v^{(n)}} f^{(n)}, \ldots, T_{\vec b_m\cdot \vec v^{(n)}} f^{(n)}) - 1| \ll_k \eps.$$

Now we invoke a local testability lemma:

\begin{lemma}[Polynomiality is locally testable]\label{polytest-lem} Let $V$ be a finite-dimensional vector space, let $k \geq 1$, let $g: V \to \D$ be a bounded function, and suppose that
\begin{equation}\label{eek-eq} \E_{h_1,\ldots,h_k \in V} \int_V |\mder_{h_1} \ldots \mder_{h_k} g - 1| \leq \eps
\end{equation}
for some $\eps>0$.  Then there exists a phase polynomial $\phi \in \Phase_{k-1}(V)$ such that
$$ \int_V |g- \phi| \leq o_{\eps \to 0;d}(1).$$
\end{lemma}

For $\F=\F_2$, this result is essentially in \cite{akklr} or \cite[Proposition 4.6]{tao-stoc}, but for the convenience of the reader (and in view of the subtle difference between phase polynomials and polynomials, see Remark \ref{fail}) we give a full proof of this lemma in the appendix.

Applying this lemma, we conclude that there exists $\phi^{(n)} \in \Phase_{k-1}(V^{(n)})$ such that
$$ \int_V |G(T_{\vec b_1\cdot \vec v^{(n)}} f^{(n)}, \ldots, T_{\vec b_m\cdot \vec v^{(n)}} f^{(n)}) - \phi^{(n)}| \leq o_{\eps \to 0;k}(1).$$
Inserting this into \eqref{c4} we conclude that
$$ |\int_V f^{(n)} \overline{\phi^{(n)}}| > c/8$$
if $\eps$ is sufficiently small depending on $c,k$.  But this contradicts \eqref{fund}.  The proof of Theorem \ref{main2} is complete.

The proof of Theorem \ref{main} is identical, but with $k$ now set equal to $d$, and Theorem \ref{ergmain-thm} used instead of Theorem \ref{ergmain-thm2}.  We leave the details to the reader.

\begin{remark} It is tempting to try to adapt these arguments to the cyclic setting $\Z/N\Z$, in which the role of polynomials is replaced by that of a \emph{nilsequence} (see \cite{gt:inverse-u3}, \cite{green-tao-linearprimes} for further discussion), thus establishing the \emph{Inverse conjecture for the Gowers norm} for $\Z/N\Z$ that was formulated in those papers.  The analogue of Theorem \ref{ergmain-thm} is known, see \cite{host-kra}.  However, two obstructions remain before one can carry out this program.  The first is to compensate for the rigidity of arithmetic progressions that seems to prevent a counterpart of Corollary \ref{sampling-2} from holding in the cyclic group setting (see Remark \ref{cyclic}).  The second is that whereas polynomiality is locally testable thanks to Lemma \ref{polytest-lem}, it is unclear whether the property of being a nilsequence is similarly testable.  
\end{remark}

\appendix

\section{Proof of Lemma \ref{polytest-lem}}

In this appendix we give a proof of Lemma \ref{polytest-lem}, following the arguments in \cite{akklr} and \cite[Proposition 4.6]{tao-stoc}.  We begin with a variant of Lemma \ref{equiv-lem}:

\begin{lemma}[Discreteness]\label{disc-lem} Let $k \geq 0$, let $V$ be a finite-dimensional vector space, and $\phi \in \Phase_k(V)$.  Then there exists $\theta \in \R/\Z$ and an integer $K \geq 1$ depending only on $\F$ such that $\phi(x)$ is equal to $e^{2\pi i \theta}$ times a $K^{\operatorname{th}}$ root of unity for every $x \in V$.
\end{lemma}

\begin{proof}  See \cite[Lemma D.5]{berg} (which gives the explicit value $K = p^{\lfloor k/p\rfloor+1}$, where $p$ is the characteristic of $\F$).
\end{proof}

We also have a rigidity lemma. 

\begin{lemma}[Rigidity]\label{rigid}   Let $k \geq 0$, let $V$ be a finite-dimensional vector space, and $\phi \in \Phase_k(V)$.  Suppose that $\int_V |\phi - 1| \leq \eps$ for some $\eps > 0$.  If $\eps$ is sufficiently small depending on $k, \F$, then $\phi$ is constant.
\end{lemma}

\begin{proof} We induct on $k$.  For $k=0$ the claim is obvious, and for $k=1$ $\phi$ is a linear character (times a phase) and the claim can be worked out by hand.  Now suppose $k \geq 2$ and the claim has already been shown for smaller values of $k$.  Since $\phi$ is a phase polynomial, we have $\mder_0 \ldots \mder_0 \phi = 1$, and thus $\phi$ has unit magnitude.  Observe that if $\int_V |\phi - 1| \leq \eps$, then $\int_V |T_h \phi - 1| \leq \eps$ for every $h \in V$.  Using the elementary estimate
$$ |\mder_h \phi - 1| \leq |\phi - 1| + |T_h \phi - 1|$$
(using the fact that $\phi$ has unit magnitude) we conclude that
$$\int_V |\mder_h \phi - 1| \leq 2\eps$$
for every $h \in V$.  On the other hand, $\mder_h \phi \in \Phase_{k-1}(V)$, so by induction hypothesis (if $\eps$ is small enough) we conclude that $\mder_h \phi$ is constant for all $h \in V$.  Thus $\phi \in \Phase_1(V)$, but then the claim follows from the $k-1$ case.
\end{proof}

We now prove Lemma \ref{polytest-lem}.  The case $k=1$ is easy, so suppose that $k \geq 2$ and the claim has already been established for $k-1$.  To abbreviate the notation we shall write $o(1)$ for $o_{\eps \to 0;k}(1)$.  We say that a statement $P(x)$ holds for \emph{most} $x \in V$ if it holds for $(1-o(1))|V|$ elements of $v$.

We fix $k, V, f$.  We may assume that $\eps$ is small depending on $d$, as the claim is trivial otherwise.  From \eqref{eek-eq} and Markov's inequality we see that
\begin{equation}\label{good-eq} \E_{h_1,\ldots,h_{k-1} \in V} \int_V |\mder_{h_1} \ldots \mder_{h_{k-1}} \mder_{h} f - 1| = o(1)
\end{equation}
for most $h \in V$.  Let us call $h$ \emph{good} if \eqref{good-eq} holds.  Applying the induction hypothesis, we conclude that for any good $h$ there exists\footnote{This quantity plays the same role that \emph{cocycles} do in ergodic theory.} $\phi_{h} \in \Phase_{k-2}(V)$ such that
$$ \int_V |\mder_{h} f - \phi_{h}| \leq o(1).$$
In particular, this implies (by Markov's inequality) that for all good $h$, we have
$$ f(x+h) \overline{f(x)} = \phi_{h}(x) + o(1)$$
for most $V$.  Since $f$ is bounded in magnitude by $1$, this implies that
$$ |f(x)| = 1- o(1)$$
for most $x$, and for all good $h$ we have
\begin{equation}\label{fxh-eq} f(x+h) = \phi_h(x) f(x) + o(1)
\end{equation}
for most $x$.

We now pause to perform a discretisation trick.  Write $p := \charac(\F)$. From repeated applications of \eqref{fxh-eq} we see that
$$ f(x) = f(x+ph) = \phi_h(x) \phi_h(x+h) \ldots \phi_h(x+(p-1)h) f(x) + o(1)$$
for most $x$, and thus
$$ \phi_h(x) \phi_h(x+h) \ldots \phi_h(x+(p-1)h) = 1 + o(1)$$
for at least one $x$.  On the other hand, from Lemma \ref{disc-lem} $\phi_h$ takes values in $e^{2\pi i \theta}$ times $K^{\operatorname{th}}$ roots of unity for some fixed $K$ depending only on $d,p$.  Thus $e^{2\pi i p \theta}$ times a $K^{\operatorname{th}}$ root of unity is within $o(1)$ of $1$, and so $e^{2\pi i \theta}$ lies within $o(1)$ of a $pK^{\operatorname{th}}$ root of unity.  Rotating $\phi_h$ by $o(1)$ if necessary we may assume that $e^{2\pi i \theta}$ is \emph{exactly} a $pK^{\operatorname{th}}$ root of unity, and in particular we have
\begin{equation}\label{kdisc-eq}
\phi_h^{pK} \equiv 1
\end{equation}
whenever $h$ is good.

Now suppose that $h_1,h_2,h_3,h_4$ are good and form an \emph{additive quadruple} in the sense that $h_1+h_2=h_3+h_4$.  Then from \eqref{fxh-eq} we see that
\begin{equation}\label{fsh}
 f(x+h_1+h_2) = f(x) \phi_{h_1}(x) \phi_{h_2}(x+h_1) + o(1)
\end{equation}
for most $x$, and similarly
$$ f(x+h_3+h_4) = f(x) \phi_{h_3}(x) \phi_{h_4}(x+h_3) + o(1)$$
for most $x$.  Since $|f(x)| = 1+o(1)$ for most $x$, we conclude the approximate cocycle relationship
$$ \phi_{h_1}(x) \phi_{h_2}(x+h_1) \overline{\phi_{h_3}(x)} \overline{\phi_{h_4}(x+h_3)} = 1 + o(1)$$
for most $x$.  In particular, the average of the left-hand side in $x$ is $1-o(1)$.  Applying Lemma \ref{rigid} (and assuming $\eps$ small enough), we conclude that the left-hand side is \emph{constant} in $x$; using the discretisation \eqref{kdisc-eq}, we conclude (again for $\eps$ small enough) that it is in fact $1$.  Thus
\begin{equation}\label{pquad-eq} \phi_{h_1}(x) \phi_{h_2}(x+h_1) = \phi_{h_3}(x) \phi_{h_4}(x+h_3) 
\end{equation}
for \emph{all} $x$ and any good additive quadruple $h_1,h_2,h_3,h_4$.

Now for any $k \in V$, define the quantity $\psi(k) \in \C$ by the formula
\begin{equation}\label{rk}
 \psi(k) := \phi_{h_1}(0) \phi_{h_2}(h_1)
\end{equation}
whenever $h_1,h_2,h_1+h_2$ are simultaenously good.  Note that the existence of such an $h_1,h_2$ is guaranteed since most $h$ are good, and \eqref{pquad-eq} ensures that the right-hand side of \eqref{rk} does not depend on the exact choice of $h_1,h_2$ and so $\psi$ is well-defined.  From \eqref{kdisc-eq} we see that $\psi$ takes values in the $pK^{\operatorname{th}}$ roots of unity, and in particular only has $O(1)$ possible values.

Now let $x \in V$ and $h$ be good.  Then, since most elements of $V$ are good, we can find good $r_1,r_2,s_1,s_2$ such that $r_1+r_2 = x$ and $s_1+s_2=x+h$.  From \eqref{fsh} we see that
$$ f(y+x) = f(y+r_1+r_2) = f(y) \phi_{r_1}(y) \phi_{r_2}(y+r_1) + o(1)$$
and
$$ f(y+x+h) = f(y+s_1+s_2) = f(y) \phi_{s_1}(y) \phi_{s_2}(y+s_1) + o(1)$$
for most $y$.  Also from \eqref{fxh-eq} we have
$$ f(y+x+h) = f(y+x) \phi_h(y+x) + o(1)$$
for most $y$.  Combining these (and the fact that $|f(y)| = 1+o(1)$ for most $y$) we see that
$$ \phi_{s_1}(y) \phi_{s_2}(y+s_1) \overline{\phi_{r_1}(y)} \overline{\phi_{r_2}(y+r_1)} \overline{\phi_h(y+x)} = 1+o(1)$$
for most $y$.  Taking expectations and applying Lemma \ref{rigid} and \eqref{kdisc-eq} as before, we conclude that
$$ \phi_{s_1}(y) \phi_{s_2}(y+s_1) \overline{\phi_{r_1}(y)} \overline{\phi_{r_2}(y+r_1)} \overline{\phi_h(y+x)} = 1$$
for \emph{all} $y$.  Specialising to $y=0$ and applying \eqref{rk} we conclude that
\begin{equation}\label{phx}
\phi_h(x) = \psi(x+h) \overline{\psi(x)} = \mder_h \psi(x)
\end{equation}
for all $x \in V$ and good $h$; thus we have succesfully ``integrated'' $\phi_h$.  We can then extend $\phi_h(x)$ to all $h \in V$ (not just good $h$) by viewing \eqref{phx} as a \emph{definition}.  Observe that if $h \in V$, then $h=h_1+h_2$ for some good $h_1,h_2$, and from \eqref{phx} we have
$$ \phi_h(x) = \phi_{h_1}(x) \phi_{h_2}(x+h_1).$$
In particular, since the right-hand side lies in $\Phase_{k-2}(V)$, the left-hand side does also.  Thus we see that $\mder_h \psi \in \Phase_{k-2}(V)$ for all $h \in V$, and thus $Q \in \Phase_{k-1}(V)$.  If we then set $g(x) := f(x) \overline{\psi}(x)$, then from \eqref{fxh-eq}, \eqref{phx} we see that for every $h \in H$ we have
$$ g(x+h) = g(x) + o(1)$$
for most $x$.  From Fubini's theorem, we thus conclude that there exists an $x$ such that $g(x+h) = g(x)+o(1)$ for most $h$, thus $g$ is almost constant.  Since $|g(x)| = 1+o(1)$ for most $x$, we thus conclude the existence of a phase $\theta \in \R/\Z$ such that $g(x) = e^{2\pi i \theta} + o(1)$ for most $x$.  We conclude that
$$ f(x) = e^{2\pi i \theta} \psi(x) + o(1)$$
for most $x$, and Lemma \ref{polytest-lem} then follows.


\begin{thebibliography}{10}

\bibitem{akklr}
N. Alon, T. Kaufman, M. Krivelevich, S. Litsyn and D. Ron, \emph{Testing
low-degree polynomials over GF(2)}, RANDOM-APPROX 2003, 188--199. Also: Testing Reed-Muller codes, IEEE Transactions on Information Theory 51 (2005), 4032--4039.

\bibitem{berg}
V. Bergelson, T. Tao and T. Ziegler, \emph{An inverse theorem for the uniformity seminorms associated with the action of $\F_p^\infty$}, preprint.

\bibitem{bv} A. Bogdanov, E. Viola, \emph{Pseudorandom bits for polynomials}, Proc. of FOCS 2007, 41--51.

\bibitem{fhk}
N. Frantzikinakis, B. Host, B. Kra, \emph{Multiple recurrence and convergence for sequences related to the prime numbers}, preprint.

\bibitem{furst}
H. Furstenberg, \emph{Ergodic behavior of diagonal measures and a theorem of Szemer\'edi on arithmetic progressions}, {J. Analyse Math.}  \textbf{31} (1977), 204--256.

\bibitem{gowers-4}
T. Gowers, \emph{A new proof of Szemer\'edi's theorem for arithmetic
progressions of length four}, Geom. Func. Anal. \textbf{8} (1998), 529--551.

\bibitem{gowers}
T. Gowers, \emph{A new proof of Szemeredi's theorem}, Geom. Func. Anal., \textbf{11} (2001), 465-588.

\bibitem{wolf}
T. Gowers, J. Wolf, \emph{The true complexity of a system of linear equations}, preprint.

\bibitem{gt-primes}
B. Green, T. Tao, \emph{The primes contain arbitrarily long arithmetic progressions}, {Annals of Math.}, Annals of Math. \textbf{167} (2008), 481--547.

\bibitem{gt:inverse-u3}
B. Green, T. Tao, \emph{An inverse theorem for the Gowers $U^3(G)$ norm}, Proc. Edin. Math. Soc. \textbf{51} (2008), 73--153.

\bibitem{green-tao-linearprimes} B. Green, T. Tao, \emph{Linear equations in primes,} {Annals of Math.}, to appear.

\bibitem{green-tao-finfieldAP4s} B. Green, T. Tao, \emph{New bounds for Szemer\'edi's Theorem, I: Progressions of length 4 in finite field geometries}, Proc. Lond. Math. Soc. \textbf{98} (2009), 365--392.

\bibitem{finrat} B. Green, T. Tao, \emph{The distribution of polynomials over finite fields, with applications to the Gowers norms}, preprint.

\bibitem{host-kra} B.~Host and B.~Kra, \emph{Nonconventional ergodic averages and nilmanifolds,}  Ann. of Math. (2)  \textbf{161}  (2005),  no. 1, 397--488.

\bibitem{lms} S. Lovett, R. Meshulam, A. Samorodnitsky, \emph{Inverse conjecture for the Gowers norm is false}, STOC 2008.

\bibitem{kauf}
T. Kaufman, S. Lovett, \emph{Worst Case to Average Case Reductions for Polynomials}, FOCS 2008: 166--175.

\bibitem{sam}
A. Samorodnitsky, \emph{Low-degree tests at large distances}, STOC 2007.

\bibitem{stv}
M. Sudan, L. Trevisan, S. Vadhan, \emph{Pseudorandom generators without the XOR lemma},
Special issue on the Fourteenth Annual IEEE Conference on Computational Complexity (Atlanta, GA, 1999). 

\bibitem{tao-quant}
T. Tao, \emph{A quantitative ergodic theory proof of Szemer\'edi's theorem}, Electron. J. Combin. \textbf{13} (2006) 1 No. 99, 1--49.

\bibitem{tao-stoc}
T. Tao, \emph{Structure and randomness in combinatorics}, Proceedings of the 48th annual symposium on Foundations of Computer Science (FOCS) 2007, 3--18.

\bibitem{tao-vu} T. Tao and V.  Vu, Additive Combinatorics, Cambridge Univ. Press, 2006.

\bibitem{varnavides} P. Varnavides, \emph{On certain sets of positive density,} J. London Math. Soc. \textbf{34} (1959) 358--360.


\end{thebibliography}
\end{document}